\documentclass[a4paper,12pt]{article}
\usepackage[english]{babel}
\usepackage[utf8x]{inputenc}
\usepackage[T1]{fontenc}
\usepackage{float}
\usepackage[a4paper,top=3.5cm,bottom=3cm,left=3cm,right=3cm,marginparwidth=1.75cm]{geometry}
\usepackage{fancyhdr}
\usepackage{amssymb,amsmath,amsthm,amsfonts}
\usepackage{graphicx}
\usepackage{subcaption}
\usepackage[colorinlistoftodos]{todonotes}
\usepackage[colorlinks=true, allcolors=blue]{hyperref}
\newtheorem{lemma}{Lemma}
\newtheorem{remark}{Remark}
\newtheorem{conjecture}{Conjecture}
\newtheorem{proposition}{Proposition}
\newtheorem{theorem}{Theorem}
\usepackage{tikz}
\usetikzlibrary{patterns}

\usetikzlibrary{arrows}
\title{The perfect $2$-colorings of infinite circulant graphs with a continuous set of odd distances\thanks{This work was performed within the framework of the LABEX MILYON (ANR-10-LABX-0070) of Universit\'{e} de Lyon, within the program ``Investissements d'Avenir'' (ANR-11-IDEX-0007) operated by the French National Research Agency (ANR), and has been supported by RFBS grant 18-31-00009. \newline Published in Siberian \`Electronic Mathematical Reports, \href{http://semr.math.nsc.ru/v17/p590-603.pdf}{Volume 17, pp. 590--603 (2020)}}}
\author{Parshina~O.~G.\textsuperscript{1} \and Lisitsyna~M.~A.\textsuperscript{2}}
\date{\small
\textsuperscript{1}Czech Technical University in Prague,\\  Trojanova 13, 120 00 Prague, Czech Republic\\
\textsuperscript{2} Marshal Budyonny Military Academy of Telecommunications,\\ Tikhoretskii pr. 3, 194064 St. Petersburg, Russia\\
\vspace{0.2cm}
}
\begin{document}

\maketitle

\begin{abstract}
A vertex coloring of a given simple graph $G=(V,E)$ with $k$ colors ($k$-coloring) is a map from its vertex set to the set of integers $\{1,2,3,\dots, k\}$. 
A coloring is called perfect if the multiset of colors appearing on the neighbours of any vertex depends only on the color of the vertex.
We consider perfect colorings of Cayley graphs of the additive group of integers with generating set $\{1,-1,3,-3,5,-5,\dots, 2n-1,1-2n\}$ for a positive integer $n$. We enumerate perfect $2$-colorings of the graphs under consideration and state the conjecture generalizing the main result to an arbitrary number of colors.

\end{abstract}

\section*{Introduction}

Let $G$ be a simple graph, $k$ be a positive integer and $M=(m_{ij})_{i,j=1}^k$ be a non-negative integer matrix of order $k$. A coloring of vertices of $G$ with $k$ colors is a map $\varphi: V \rightarrow \{1, 2, 3, \dots, k\}$. The value $\varphi(v) = s$ is said to be the {\it color} of $v$. 
Hereinafter by coloring of a graph we mean a coloring of its vertex set.
A coloring of the graph $G$ is called {\it perfect} with parameter matrix $M$, if for any integers $i,j$ in range from 1 to $k$ any vertex colored with $i$ has exactly $m_{ij}$ neighbors colored with $j$. 
In this case the matrix $M$ is called \textit{admissible} for the graph $G$.
The corresponding partition of the vertex set of $G$ is known as 
equitable. 

The concept of perfect coloring plays an important role in graph theory, algebraic combinatorics and coding theory.
The notion of perfect coloring is closely related to the notion of perfect code.
{For example, a distance partition of a distance regular graph in accordance to a perfect code is a perfect coloring.}

Hereinafter $n$ and $k$ are positive integers.
In this paper we aim to classify perfect colorings of graphs from the family of infinite circulants.
The graphs under consideration are Cayley graphs of the additive group of integers with generating set $\{1,-1,3,-3,5,-5,\dots, 2n-1, 1-2n\}$.
We call such graphs infinite circulant graphs with the set of distances $\{1,3,5,\dots, 2n-1\}$.
Perfect $2$-colorings of infinite circulant graphs with the set of distances $\{1,2,3,\dots, n\}$ are enumerated in \cite{P2}. The conjecture generalizing described result to the case of arbitrary number of colors is posed in \cite{Pk}. Partial results on the conjecture one can see in \cite{LP12}.

{Closely related to infinite circulant graphs is the $n$-dimensional rectangular grid graph $G(\mathbb Z^n)$, which is a covering graph of any infinite circulant with $n$ distances.}
Perfect colorings of the infinite rectangular grid graph have been widely studied. Admissible for the graph $G(\mathbb Z^2)$ parameter matrices of order 3 are enumerated by S.~A.~Puzynina~\cite{Pdis}. Perfect $s$-colorings of the graph $G(\mathbb Z^2)$ for $s\leq9$ are listed by D.~S.~Krotov~\cite{KrArxiv}.

A perfect $k$-coloring is called \textit{distance regular} if its parameter matrix is tridiagonalizable. In this case colors of the coloring can be arranged in a way that every vertex of color $i\in\{2,3,\dots,k-1\}$ can only be adjacent to vertices of color $i-1$, $i$ and $i+1$. Moreover, the set of vertices of color $1$ and the set of vertices of color $k$ are completely regular codes.
Parameters of distance regular colorings of the infinite rectangular grid graph are enumerated by S.~V.~Avgustinovich, A.~Yu.~Vasil'eva and I.~V.~Sergeeva~\cite{AVSsq}.

Along with perfect colorings of the infinite rectangular grid graph, perfect colorings of triangle and hexagonal infinite grid graphs have been studied. 
S.~A.~Puzynina proved that for every perfect coloring of infinite triangle or hexagonal grids there exists a periodic coloring of the grid with the same parameter matrix~\cite{Pgrid}. 
Distance regular colorings of the infinite triangle grid graph are enumerated by A.~Yu.~Vasil'eva~\cite{VasGrid}, of the hexagonal grid graph are listed by S.~V.~Avgustinovich, D.~S.~Krotov and A.~Yu.~Vasil'eva~\cite{AKVcodes}.

Let $G=(V,E)$ be a simple graph, $M=(m_{ij})_{i,j=1}^k$ be a square matrix of order $k$, and $r\geq1$. A coloring of the vertex set of the graph $G$ is called perfect of radius $r$ with parameter matrix $M$ if the element $m_{ij}$ stands for the number of vertices of color $j$ at the distance at most $r$ from any vertex of color $i$ for each $i,j\in\{1,2,3,\dots, k\}$.

Admissible parameter matrices of perfect $2$-colorings of radius 1 of the graph $G(\mathbb Z^2)$ are enumerated by M.~Axenovich~\cite{Axe}. In the same paper the author states several necessary conditions on a parameter matrix to be admissible for $G(\mathbb Z^2)$ in the case $r\geq2$.
Parameters and properties of perfect colorings of $G(\mathbb Z^2)$ have been studied by S.~A.~Puzynina in her PhD thesis~\cite{Pdis}. 
In particular, she showed that all perfect colorings of radius $r>1$ of this graph are periodic.
Several results on perfect $2$-colorings of circulant graphs were obtained by D.~B.~Khoroshilova~\cite{Khor2,Khor1}.

Let us mention several results on perfect colorings of graphs with similar to circulants and infinite grid graphs local structure.

Perfect $2$-colorings of the hypercube graph have been studied by D.~G.~Fon-Der-Flaass. 
He obtained necessary conditions on parameters of perfect $2$-colorings of this graph and presented an infinite series of such colorings~\cite{FFHyp}.
Later he obtained a bound on correlation immunity of non-constant unbalanced Boolean functions that allows to obtain a necessary condition for a perfect coloring with given parameters to exist in the hypercube graph~\cite{FFci}. 
Fon-Der-Flaass constructed perfect colorings of the 12-dimensional hypercube graph that attain this bound~\cite{FF12}.
Another method to construct perfect $2$-colorings via parameter matrices was provided by D.~G.~Fon-Der-Flaass and K.~V.~Vorobev~\cite{FFVh}. {A new necessary condition on parameters of perfect $2$-colorings of the hypercube graph was obtained in the recent joint work of D.~S.~Krotov and K.~V.~Vorobev}~\cite{KrVor}. 
Let us note that the set of parameter matrices admissible for this graph has not been described yet even for the case of two colors.

A Johnson graph $J(n,\omega)$ is the graph with the set of boolean vectors of weight $\omega$ as set of vertices; two vertices are adjacent in $J(n,\omega)$, if they differ in exactly two coordinates. W.~J.~Martin showed that the coloring of $J(n,\omega)$ obtained by coloring vertices of blocks of $(\omega-1)-(n,\omega,\lambda)$-scheme with color 1 and all the other vertices of $J(n,\omega)$ with the color $2$ is perfect~\cite{Martin}. 

A systematic study of perfect $2$-colorings in Johnson graphs is performed in the thesis of I.~Yu.~Mogilnykh~\cite{MogJ}. 
He constructed several series of perfect $2$-colorings of Johnson graphs and provided several necessary conditions for such colorings to exist.
These results were used in enumeration of parameters of perfect $2$-colorings of Johnson graphs $J(n,\omega)$, where $n\leq8$.
In~\cite{GGJ} one can find the complete description of admissible parameter matrices of order 2 for the graph $J(n,3)$, where $n$ is odd. The problem of perfect colorings of Johnson graphs classification is not solved even in the case of two colors.

Perfect $2$-colorings of transitive cubic graphs with the set of vertices of cardinality up to 18 are enumerated by S.~V.~Avgustinovich and M.~A.~Lisitsyna in~\cite{AvgLisCube}. In the later work the authors listed perfect colorings of the infinite prism graph with arbitrary number of colors~\cite{LisAvgPr}.

\section{Preliminaries}

Let $G=(V,E)$ be a graph with vertex set $V$ and edge set $E$. 
For a given vertex $v\in V$, we denote the set of vertices adjacent to $v$ by $N(v)$ and call it the {\it neighborhood} of $v$.

We are interested in graphs defined as follows. Let us consider a set $D=\{d_1,d_2,\dots,d_n\}$ of positive integers enumerated in ascending order. 
We say that the graph $\mathrm{Ci}_\infty(D)=(\mathbb Z,E)$, where $E=\{(i,i\pm d)| i\in\mathbb Z, d\in D\}$, is the {\it infinite circulant graph with the set of distances D}. 
This graph can be regarded as Cayley graph of the additive group of $\mathbb Z$ with the generating set $\{\pm d_j\}_{j=1}^n$. 
Along with infinite circulant graphs we consider finite ones. 
Let $t$ be a positive integer. 
A {\it finite circulant graph with the set of distances $D$} is the graph $\mathrm{Ci}_t(D)$ with the set $\mathbb Z_t$ as the vertex set and the multiset $\{(i, i+d$~mod~$t) \,|\, i\in\mathbb Z_t, d\in D \}$ as the edge set. 
Such graphs can have multiedges and loops, namely they are pseudographs. 
A coloring of a pseudograph is called perfect if for two vertices of the same color the multisets of colors of their neighborhoods coincide. 
By the multiset of colors of a vertex $v$ neighborhood we mean the multiset where the number of occurrences of a color $i$ is equal to the number of edges between the vertex $v$ and vertices of color $i$.

{Let $G_1= (V_1, E_1)$ and $G_2 = (V_2, E_2)$ be two pseudographs. A surjection $f: V_1 \rightarrow V_2$ is a \textit{covering map} from $G_1$ to $G_2$ if for each vertex $v\in V_1$, the restriction of $f$ to the neighbourhood of $v$ is a bijection onto the neighbourhood of $f(v)$ in $G_2$. In other words, $f$ maps edges incident to $v$ one-to-one onto edges incident to $f(v)$.
If there exists a covering map from $G_1$ to $G_2$, then $G_1$ is a \textit{covering graph} of $G_2$.}

\begin{proposition}\label{h}
{Let $G_1$ and $G_2$ be pseudographs. If there exists a covering map from $G_1$ to $G_2$, then every perfect coloring of $G_2$ induces a perfect coloring of $G_1$ with the same parameter matrix.}
\end{proposition}
{The proof of this statement follows immediately from the definitions of covering map and perfect coloring. }

Proposition \ref{h} provides us a method of constructing perfect colorings of a given graph using perfect colorings of other graphs, which are usually chosen to have more convenient for this purpose structure.
{We will use a covering map from $\mathrm{Ci}_\infty(D)$ to a finite pseudograph $\mathrm{Ci}_t(D)$ in enumeration of perfect colorings of the graph under consideration. }

Let $t$ be a positive integer. A coloring $\varphi$ of the circulant graph $\mathrm{Ci}_\infty(D)$ is {\it periodic} with the length of period $t$, if $\varphi(i)=\varphi(i+t)$ for every $i\in\mathbb{Z}$. 
We  will write $[\varphi(i+1)\varphi(i+2)\cdots\varphi(i+t)]$ to depict the period of $\varphi$.

Hereinafter $D_n$ stands for the set of distances $\{1,3,5,\dots, 2n-1\}$. In the paper we consider finite and infinite circulants with the set of distances $D_n$. In finite case we are interested in circulants with even number of vertices. 

Let us call graphs $\mathrm{Ci}_\infty(D_n)$ and $\mathrm{Ci}_t(D_n)$, $t\in2\mathbb{N}$, {\it infinite and finite circulant graphs with a continuous set of odd distances} respectively. 
These graphs are regular of degree $2n$ and bipartite. 
For a given graph $\mathrm{Ci}_l(D_n)$, where $l\in2\mathbb{N}\cup\{\infty\}$, we denote by $V_e$ the set of its vertices with even indices, and by $V_o$ the set of vertices with odd indices.

We will write $v_e$ or $v_o$ when it is necessary to indicate that a vertex belongs to even or to odd part of the graph respectively.

\begin{proposition}\label{period}
Every perfect coloring of the graph $\mathrm{Ci}_\infty(D_n), n\in\mathbb{N}$, is periodic.
\end{proposition}
\begin{proof}
Let $\varphi$ be a perfect coloring of $\mathrm{Ci}_\infty(D_n)$ with parameter matrix $M$. 
Let us take an arbitrary integer $i$ and consider a vertex $v_i$ with its neighborhood $N(v_i)=v_{i-2n+1}v_{i-2n-1}v_{i-2n-3}\cdots v_{i-3}v_{i-1} v_{i+1}v_{i+3}\cdots v_{i+2n-1}$ perfectly colored with $\varphi$. 
Let us consider the vertex $v_{i+2n+1}$. Since it is the only vertex from the set $N(v_{i+2})\backslash N(v_i)$, its color is uniquely determined by the color of the vertex $v_{i+2}$ and the parameter matrix $M$. 
The same holds for the vertex $v_{i-2n+3}$ by symmetry.
This property induces the periodicity of the coloring $\varphi$.
\end{proof}

We say that a coloring $\varphi$ of a bipartite graph is {\it bipartite} if sets of colors of the even and odd parts of the graph are disjoint. 

\begin{remark}\label{even}
\normalfont
{Let $\varphi$ be a periodic perfect coloring of a bipartite graph. Then either $\varphi$ is bipartite, or the even and the odd parts of the graph contain the same number of vertices of every color.}
\end{remark}
{This remark gives the necessary condition for a perfect coloring to exist in the graphs under consideration. }

The following proposition concerns perfect colorings of the infinite path graph, which is, in our terms, the infinite circulant graph $\mathrm{Ci}_\infty(\{1\})$.\\~\\
\begin{proposition}
Let $k$ be a positive integer. The list of perfect $k$-colorings of the graph $\mathrm{Ci}_\infty(\{1\})$ is exhausted by colorings with the following four periods:
\begin{enumerate}
\item $[123\cdots (k-1)k]$;
\item $[k(k-1)(k-2)\cdots212\cdots (k-2)(k-1)]$;
\item $[k(k-1)(k-2)\cdots212\cdots (k-2)(k-1)k]$;
\item $[k(k-1)(k-2)\cdots2112\cdots (k-2)(k-1)k]$.
\end{enumerate}
\end{proposition}
The proof of this statement can be found, for example, in \cite{LisAvgPr} (Lemma~2). Let us note that these colorings are perfect for every infinite circulant graph under consideration.  

We state the following conjecture.
\begin{conjecture}\label{conj}
Let $k$ and $n$ be positive integers. The set of perfect $k$-colorings of the graph $\mathrm{Ci}_\infty(D_n)$ consists of perfect colorings induced from perfect colorings of the infinite path graph and of graphs $\mathrm{Ci}_t(D_n)$ for $t=4n-2,4n,4n+2$.
\end{conjecture}

In this paper we prove the conjecture for $k=2$. In this case the set of perfect colorings of the infinite path graph consists of three equivalence classes of colorings with periods $[12]$, $[212]$ and $[2112]$.

\subsection{Perfect colorings of finite bipartite circulants}\label{finite}

In this section we consider perfect colorings of graphs $\mathrm{Ci}_t(D_n)$ for $t\in\{4n,4n-2,4n+2\}$, $n\in\mathbb{N}$.

\subsubsection{The case $t=4n$}

The graph $\mathrm{Ci}_{4n}(D_n)$ is the complete bipartite graph $K_{2n,2n}$. A coloring of this graph is perfect if it is bipartite or if odd and even parts of the graph contain the same number of vertices of each color (see Remark~\ref{even}). 

{To construct a bipartite perfect coloring of this graph, we should split the set of colors into two disjoint subsets $C_e$ and $C_o$, and then color vertices of the even (odd) part of the graph with colors from $C_e$ ($C_o$) in arbitrary order. It is easy to see that every coloring obtained this way is perfect for $K_{2n,2n}$.

For any perfect and non-bipartite $k$-coloring of this graph, the neighborhood of every vertex has the same coloring structure regardless its own color, thus column elements in the parameter matrix of any such coloring are equal. In other words, for every index $j\in\{1,2,\dots,k\}$, $m_{ij}=m_{kj}:=m_j$ $\forall i,k\in\{1,2,\dots,k\}$. 

In this case the number of vertices of each color $j$ in each part of the graph must be equal to $m_j$, and we can color each part of graph independently, putting colors in arbitrary order.
For a given parameter matrix there exist $2\frac{k}{m_1!m_2!\cdots m_k!}$ different non-bipartite $k$-colorings of $\mathrm{Ci}_{4n}(D_n)$. }

Figure~\ref{Fig:Ci8} shows the graph $\mathrm{Ci}_{8}(\{1,3\})$ perfectly colored with three colors. 
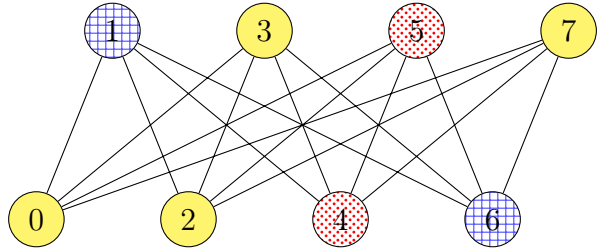
\begin{figure}[H]
\centering
\begin{tikzpicture}
\centering
\tikzstyle{every node}=[draw,shape=circle];
\node (0) [fill=yellow!70] at (0,0) {0};
\node (1) [pattern=grid, pattern color=blue!70] at (1,2.5) {1};
\node (2) [fill=yellow!70] at (2,0) {2};
\node (3) [fill=yellow!70] at (3,2.5) {3};
\node (4) [pattern=crosshatch dots, pattern color=red] at (4,0) {4};
\node (5) [pattern=crosshatch dots, pattern color=red] at (5,2.5) {5};
\node (6) [pattern=grid, pattern color=blue!70] at (6,0) {6};
\node (7) [fill=yellow!70] at (7,2.5) {7};

\draw (0) -- (1)
      (0) -- (3)
      (0) -- (5)
      (0) -- (7)
	  (2) -- (1)
      (2) -- (3)
      (2) -- (5)
      (2) -- (7)
      (4) -- (3)
      (4) -- (5)
      (4) -- (1)
      (4) -- (7)
      (6) -- (3)
      (6) -- (5)
      (6) -- (7)
      (6) -- (1);
\end{tikzpicture}
\caption{Perfect $3$-coloring of the graph $\mathrm{Ci}_{8}(\{1,3\})$.}
\label{Fig:Ci8}
\end{figure}

\subsubsection{The case $t=4n+2$}\label{4n+2}

Let us remind that a {\it perfect matching} of a graph is an independent edge set in which every vertex of the graph is incident to exactly one edge of the matching.

We may say that the graph $\mathrm{Ci}_{4n+2}(D_n)=(V_e\cup V_o,E)$ is the complete bipartite graph $K_{2n+1,2n+1}$ without the perfect matching $P_{2n+1}=\{(i,i+2n+1)| i=0,1,2,\dots,2n\}$. 
In other words, every vertex $i$ of one part of the graph is adjacent to all vertices of another part except for the vertex $j$ such that $(i,j)\in P_{2n+1}$. 

Let $\varphi$ be a perfect coloring of $\mathrm{Ci}_{4n+2}(D_n)$. 
Let us consider an edge $(i,j)$ from $P_{2n+1}$, its endpoints $i$ and $j$ are colored with (not necessary distinct) colors $\varphi(i)$ and $\varphi(j)$. 
It is easy to see that in this case any edge from $P_{2n+1}$ having one endpoint colored with $\varphi(i)$ must have another endpoint colored with $\varphi(j)$, and vice versa. 
{This condition directly follows from the definition of perfect coloring and means, in particular, that the set of colors of a bipartite perfect coloring must be of even cardinality. We will use this necessary condition to construct perfect colorings, let us refer to it as the condition $(\star)$.}

To construct a perfect bipartite coloring $\varphi$ of $\mathrm{Ci}_{4n+2}(D_n)$, we split the set of colors $C$ into two disjoint subsets of the same cardinality $C=C_e\cup C_o$; 
then we arrange colors in pairs $(c_i,c_j)$, where $i,j\in\{1,2,\dots, |C_e|\}$, and $c_i\in C_e$, $c_j\in C_o$;
we color every edge $(v_e,v_o)$ of $P_{2n+1}$ with one of the assigned pairs of colors such that $\varphi(v_e)\in C_e$, $\varphi(v_o)\in C_o$. 

Let us construct a non-bipartite perfect coloring. By the definition of perfect coloring, each part of the graph must have the same number of vertices of each color. This, together with the condition $(\star)$, give us the following method.
Let $\phi$ denote the coloring we are going to construct.

Endpoints of every edge can be colored with the same color or differently. From the conditions above it follows, that if there is an edge $(v_e,v_o)$ with $\phi(v_e)=x$ and $\phi(v_o)=y$, $x\neq y$, then there must be an edge $(u_e,u_o)$ with $\phi(u_e)=y$ and $\phi(u_o)=x$, otherwise the coloring $\phi$ cannot be perfect.

Let us split the set of colors into two disjoint subsets, $C=C_1\cup C_2$, where $|C_2|$ is even.
Along with that we split the edges of the perfect matching $P_{2n+1}$ into two disjoint subsets $P_1$ and $P_2$, where $|P_2|$ is even. We color the endpoints of edges from the set $P_1$ with colors from the set $C_1$ in a way that endpoints of every edge get the same color. 

We arrange colors of $C_2$ and edges from $P_2$ in pairs. 
We color each pair of edges $(v_e, v_o),(u_e,u_o)$ of the set $P_2$ in a way that endpoints of every edge get different colors, but $\phi(v_e)=\phi(u_o)$, and $\phi(v_o)=\phi(u_e)$.

If the set $C_2$ is empty, then each edge has endpoints colored with the same color. The set $C_1$ can be empty only if the edge set is of even cardinality, then all edges belong to $P_2$ and colored in a way described above. It is easy to verify that in both cases colorings will be perfect.

{It should be noted that this construction follows only from the necessary conditions on a coloring of the bipartite graph to be perfect and non-bipartite, and every perfect coloring of such graph can be obtained using this procedure.}

An example of a perfect non-bipartite coloring is shown in Figure~\ref{Fig:Ci10}. The absent perfect matching is $P_{10}=\{(0,5), (1,6), (2,7), (3,8), (4,9)\}=P_1\cup P_2$, where $P_1=\{(0,5),(1,6),(3,8)\}$ and $P_2=\{(2,7),(4,9)\}$. 

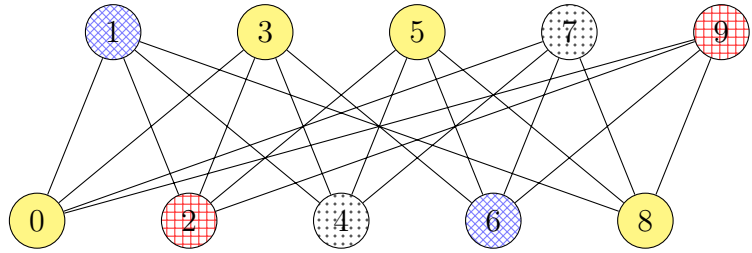
\begin{figure}[H]
\centering
\begin{tikzpicture}
\centering
\tikzstyle{every node}=[draw,shape=circle];
\node (0) [fill=yellow!60] at (0,0) {0};
\node (1) [pattern=crosshatch, pattern color=blue!50] at (1,2.5) {1};
\node (2) [pattern=grid, pattern color=red!80] at (2,0) {2};
\node (3) [fill=yellow!60] at (3,2.5) {3};
\node (4) [pattern= dots, pattern color=gray!130] at (4,0) {4};
\node (5) [fill=yellow!60] at (5,2.5) {5};
\node (6) [pattern=crosshatch, pattern color=blue!50] at (6,0) {6};
\node (7) [pattern= dots, pattern color=gray!130] at (7,2.5) {7};
\node (8) [fill=yellow!60] at (8,0) {8};
\node (9) [pattern=grid, pattern color=red!80] at (9,2.5) {9};

\draw (0) -- (1)
      (0) -- (3)
      (0) -- (9)
      (0) -- (7)
	  (2) -- (1)
      (2) -- (3)
      (2) -- (9)
      (2) -- (5)
      (4) -- (3)
      (4) -- (5)
      (4) -- (1)
      (4) -- (7)
      (6) -- (3)
      (6) -- (5)
      (6) -- (7)
      (6) -- (9)
      (8) -- (1)
      (8) -- (9)
      (8) -- (7)
      (8) -- (5);
\end{tikzpicture}
\caption{Perfect $4$-coloring of the graph $\mathrm{Ci}_{10}(\{1,3\})$.}
\label{Fig:Ci10}
\end{figure}

\subsubsection{The case $t=4n-2$}\label{4n-2}

Let us consider the perfect matching on $4n-2$ vertices $P_{2n-1}=\{(i,i+2n-1)| i=0,1,2,\dots,2n-2\}$.
Every vertex $i\in V_e$ of the bipartite pseudograph $\mathrm{Ci}_{4n-2}(D_n)=(V_e\cup V_o,E)$ is adjacent to all vertices of $V_o$ and has an extra edge to the vertex $j$ such that $(i,j)\in P_{2n-1}$. 
The same holds for every vertex of $V_o$. 
Informally speaking, $\mathrm{Ci}_{4n-2}(D_n)$ is the complete bipartite graph $K_{2n-1,2n-1}$ with extra perfect matching $P_{2n-1}$.

The coloring procedure for this graph is the same as the coloring procedure for the graph $\mathrm{Ci}_{4n-2}(D_n)$.  One should split the set of colors into two disjoint subsets and then color endpoints of edges of the perfect matching $P_{2n-1}$ in the same way as we colored edges of $P_{2n+1}$ from the previous case.

Two examples of perfect $2$-colorings of $\mathrm{Ci}_6(\{1,3\})$ are shown in figure~\ref{Fig:Ci6}. In the first case the set of colors $C=\{blue,red\}$ coincides with the set $C_1$, while sets $C_2$ and $P_2$ are empty. In the second picture the bipartite coloring is shown.

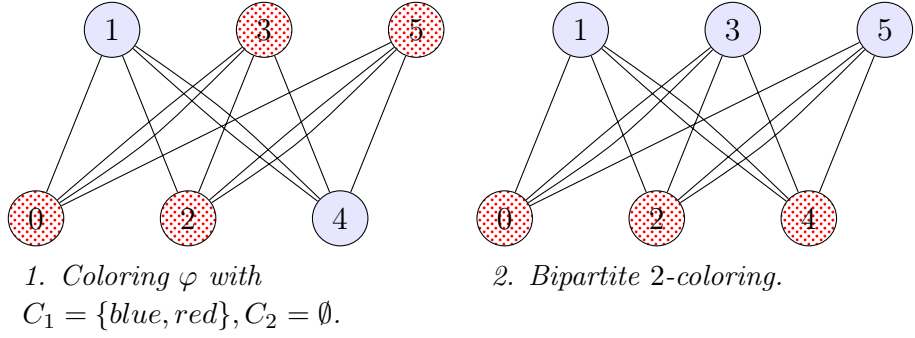
\begin{figure}[H]
\centering
\begin{tabular}{l l}
\begin{tikzpicture}
%\centering
\tikzstyle{every node}=[draw,shape=circle];
\node (0) [pattern=crosshatch dots, pattern color=red] at (0,0) {0};
\node (1) [fill=blue!10] at (1,2.5) {1};
\node (2) [pattern=crosshatch dots, pattern color=red] at (2,0) {2};
\node (3) [pattern=crosshatch dots, pattern color=red] at (3,2.5) {3};
\node (4) [fill=blue!10] at (4,0) {4};
\node (5) [pattern=crosshatch dots, pattern color=red] at (5,2.5) {5};

\draw (0) -- (1)
	  (0) .. controls(1.5,1).. (3)
      (0) .. controls(1.5,1.3).. (3)
      (0) -- (5)
	  (2) -- (1)
      (2) -- (3)
      (2) .. controls(3.5,1).. (5)
      (2) .. controls(3.5,1.3).. (5)
      (4) -- (3)
      (4) .. controls(2.5,1.45).. (1)
      (4) .. controls(2.5,1.15).. (1) 
      (4) -- (5)      ;
\end{tikzpicture}
&
\begin{tikzpicture}
\tikzstyle{every node}=[draw,shape=circle];
\node (0) [pattern=crosshatch dots, pattern color=red] at (0,0) {0};
\node (1) [fill=blue!10] at (1,2.5) {1};
\node (2) [pattern=crosshatch dots, pattern color=red] at (2,0) {2};
\node (3) [fill=blue!10] at (3,2.5) {3};
\node (4) [pattern=crosshatch dots, pattern color=red] at (4,0) {4};
\node (5) [fill=blue!10] at (5,2.5) {5};

\draw (0) -- (1)
	  (0) .. controls(1.5,1).. (3)
      (0) .. controls(1.5,1.3).. (3)
      (0) -- (5)
	  (2) -- (1)
      (2) -- (3)
      (2) .. controls(3.5,1).. (5)
      (2) .. controls(3.5,1.3).. (5)
      (4) -- (3)
      (4) .. controls(2.5,1.45).. (1)
      (4) .. controls(2.5,1.15).. (1) 
      (4) -- (5)      ;
\end{tikzpicture}\\
\raggedright 
{\small \it \enspace 1. Coloring $\varphi$ with} &
{\small \it \enspace 2. Bipartite $2$-coloring.}\\
{\small\it \enspace $C_1=\{blue,red\},C_2=\emptyset$.} & 
\end{tabular}
\caption{Perfect $2$-colorings of the graph $\mathrm{Ci}_{6}(\{1,3\})$.}
\label{Fig:Ci6}
\end{figure}

\begin{remark}\label{twoc}
\normalfont
If for $t=4n\pm2$ the set of colors is $C=\{0,1\}$, there are only two possibilities:
\begin{enumerate}
\item $C_1=C=\{0,1\}$, $C_2=\emptyset$. Endpoints of every edge of the perfect matching are either both colored with 0, or both colored with 1.
\item $C_1=\emptyset$, $C_2=\{0,1\}$. The only possible perfect $2$-coloring in this case is the bipartite one.
\end{enumerate}
\end{remark}
\section{Main result}
In this section we consider perfect $2$-colorings of the graph $\mathrm{Ci}_\infty(D_n)$. 
As a matter of convenience we will name colors of $2$-colorings {\it black} $(\bullet)$ and {\it white} $(\circ)$.
The parameter matrix of a perfect $2$-coloring has the following form: $\left (\begin{array}{cc} a & b\\ c & d\end{array}\right)$.
Since the graph under consideration is regular of degree $2n$, the parameters $a$ and $d$ can be represented as $2n-b$ and $2n-c$ respectively. 
Sometimes instead of considering the parameter matrix of a coloring we will take into account parameters $b$ and $c$, which are called {\it outer degrees} of black and white color respectively. 
A pair $(b,c)$ is called {\it admissible} for the graph $\mathrm{Ci}_\infty(D_n)$ if there exists a perfect $2$-coloring of $\mathrm{Ci}_\infty(D_n)$ with parameter matrix $\left (\begin{array}{cc} 2n-b & b\\ c & 2n-c\end{array}\right)$.

\begin{theorem}\label{thm}
Let $n$ be a positive integer, and $\mathrm{Ci}_\infty(D_n)$ be the infinite circulant graph with a continuous set of odd distances. The set of perfect $2$-colorings of $\mathrm{Ci}_\infty(D_n)$ consists of perfect colorings induced by perfect colorings of the infinite path graph and of graphs $\mathrm{Ci}_t(D_n)$ for $t=4n-2,4n,4n+2$.
\end{theorem}

Let us state and prove several preliminary lemmas.

\begin{lemma}\label{bc}
Let $n$ be a positive integer. A pair of positive integers $(b,c)$ is admissible for the graph $\mathrm{Ci}_\infty(D_n)$ if and only if $b+c\in\{4n,2n,2n+1,2n-1\}$.
\end{lemma}
\begin{proof}
The parameter matrix of the bipartite coloring of $\mathrm{Ci}_\infty(D_n)$ is $\left (\begin{array}{cc} 0 & 2n\\ 2n & 0\end{array}\right)$, and $b+c=4n$. The period of this coloring is $[\bullet\circ]$.

Let $\varphi$ be a perfect coloring of $\mathrm{Ci}_\infty(D_n)$ with period length longer than $2$. That means there exists a positive integer $i$ such that $\varphi(v_i)\neq\varphi(v_{i+2})$. 
Without loss of generality let $\varphi(v_i)=\bullet$.
The neighborhoods $N(v_i)$ and $N(v_{i+2})$ share $2n-2$ vertices, and the following holds: $N(v_i)\backslash N(v_{i+2})=\{v_{i-2n+1}\}$, $N(v_{i+2})\backslash N(v_{i})=\{v_{i+2n+1}\}$. 
Let us consider the pair of vertices $(v_{i-2n+1},v_{i+2n+1})$ and their possible colors. 
\begin{enumerate}
\item If $\varphi(v_{i-2n+1})=\varphi(v_{i+2n+1})$, neighborhoods $N(v_i)$ and $N(v_{i+2})$ have the same number of black and white vertices. 
That means every vertex is adjacent to the same number of black and white vertices regardless of its own color. In this case the parameter matrix of the coloring is $\left (\begin{array}{cc} c & b\\ c & b\end{array}\right)$, and $b+c=2n$. 
Moreover, vertices $v_{i-2n+1}$ and $v_{i+2n+1}$ are of the same color for every $i\in\mathbb{Z}$, what means the coloring $\varphi$ is periodic with the period length $4n$. 

\item Let $(\varphi(v_{i-2n+1}),\varphi(v_{i+2n+1}))=(\circ,\bullet)$. In this case every black vertex has one more white vertex in its neighborhood than the white one.
The parameter matrix of the coloring is $\left (\begin{array}{cc} c-1 & b\\ c & b-1\end{array}\right)$, and $b+c=2n+1$.

\item Let $(\varphi(v_{i-2n+1}),\varphi(v_{i+2n+1}))=(\bullet,\circ)$. In this case every black vertex has one more black vertex in its neighborhood than the white one.
The parameter matrix of the coloring is $\left (\begin{array}{cc} c+1 & b\\ c & b+1\end{array}\right)$, which means $b+c=2n-1$.
\end{enumerate}
Since all possibilities are listed, then $b+c\in\{4n,2n,2n+1,2n-1\}$.
\end{proof}

\begin{lemma}\label{xy}
Let $n,b,c$ be positive integers, and the pair $(b,c)$ be admissible for the graph $\mathrm{Ci}_\infty(D_n)$. Let $\varphi$ be a perfect $2$-coloring of this graph with $b$ and $c$ being outer degrees of black and white colors respectively. Then for every $i\in\mathbb{Z}$ the following holds:
\begin{enumerate}
\item If $\varphi(i)=\varphi(i+2)$, then $\varphi(i-2n+1)=\varphi(i+2n+1)$;\label{xy0}
\item If $\varphi(i)\neq\varphi(i+2)$ and $b+c=2n+1$, then $\varphi(i-2n+1)=\varphi(i+2)$, and $\varphi(i+2n+1)=\varphi(i)$; \label{xy+}
\item If $\varphi(i)\neq\varphi(i+2)$ and $b+c=2n-1$, then $\varphi(i-2n+1)=\varphi(i)$ and $\varphi(i+2n+1)=\varphi(i+2)$.\label{xy-}
\end{enumerate}
\end{lemma}
\begin{proof}
The proof of the lemma follows directly from the definition of perfect coloring and the proof of  Lemma~\ref{bc}.
\end{proof}

The coloring patterns provided by Lemma~\ref{xy} are depicted at the Figure~\ref{Fig:lemmaxy}.

\begin{figure}[H]
\centering
\begin{tikzpicture}
\centering
\draw 
(4,0) node [draw,shape=circle,minimum width=19pt](i){$x$} 
(4,-0.3)node[below] {\tiny $i$}
(5,0) node [draw,shape=circle,minimum width=19pt](i+2){$x$} 
(5,-0.3)node[below] {\tiny $i+2$}
(2,2.5) node [draw,shape=circle, minimum width=19pt](i-2n+3){} 
(2.2,2.8)node[above] {\tiny $i-2n+3$}
(1,2.5) node [draw,shape=circle,minimum width=19pt](i-2n+1){$y$} 
(0.8,2.8)node[above] {\tiny $i-2n+1$}
(8,2.5) node [draw,shape=circle,minimum width=19pt](i+2n+1){$y$} 
(8.2,2.8)node[above] {\tiny $i+2n+1$}
(7,2.5) node [draw,shape=circle,minimum width=19pt](i+2n-1){} 
(6.8,2.8)node[above] {\tiny $i+2n-1$}
;
\draw[black, dashed]
 (6,2.5) node [draw,shape=circle,minimum width=19pt](i-2n+5){} 
  (3,2.5) node [draw,shape=circle,minimum width=19pt](i+2n-3){} 
;
\draw[loosely dotted] 
    (4.1,2.5) -- (4.9,2.5)
    (8.5,2.5) -- (8.9,2.5)
    (1.1,0) -- (1.9,0)
    (7.1,0) -- (7.9,0)
    (0.1,2.5) -- (0.5,2.5)
;
\draw [gray!130]
      (i+2) -- (i-2n+3)
      (i+2) -- (i+2n+1)
      (i+2) -- (i+2n-1)
      (i) -- (i-2n+1)
      (i) -- (i+2n-1)
      (i) -- (i-2n+3)
      (i) -- (i-2n+5)
      (i) -- (i+2n-3)
      (i+2) -- (i-2n+5)
      (i+2) -- (i+2n-3)
;
\end{tikzpicture}\\
%\raggedright 
{\normalsize \it \enspace 1) Lemma~\ref{xy}, Item~\ref{xy0}: $x,y\in\{0,1\}$ }\\~\\
\begin{tikzpicture}
\centering
\draw 
(4,0) node [draw,shape=circle,minimum width=19pt,color=blue](i){$x$} 
(4,-0.3)node[below] {\tiny $i$}
(5,0) node [draw,shape=circle,minimum width=19pt,color=red](i+2){$\overline x$} 
(5,-0.3)node[below] {\tiny $i+2$}
(2,2.5) node [draw,shape=circle, minimum width=19pt](i-2n+3){} 
(2.2,2.8)node[above] {\tiny $i-2n+3$}
(1,2.5) node [draw,shape=circle,minimum width=19pt,color=red](i-2n+1){$\overline x$} 
(0.8,2.8)node[above] {\tiny $i-2n+1$}
(8,2.5) node [draw,shape=circle,minimum width=19pt,color=blue](i+2n+1){$x$} 
(8.2,2.8)node[above] {\tiny $i+2n+1$}
(7,2.5) node [draw,shape=circle,minimum width=19pt](i+2n-1){} 
(6.8,2.8)node[above] {\tiny $i+2n-1$}
;
\draw[black, dashed]
 (6,2.5) node [draw,shape=circle,minimum width=19pt](i-2n+5){} 
  (3,2.5) node [draw,shape=circle,minimum width=19pt](i+2n-3){} 
;
\draw[loosely dotted] 
    (4.1,2.5) -- (4.9,2.5)
    (8.5,2.5) -- (8.9,2.5)
    (1.1,0) -- (1.9,0)
    (7.1,0) -- (7.9,0)
    (0.1,2.5) -- (0.5,2.5)
;
\draw [gray!130]
      (i+2) -- (i-2n+3)
      (i+2) -- (i+2n+1)
      (i+2) -- (i+2n-1)
      (i) -- (i-2n+1)
      (i) -- (i+2n-1)
      (i) -- (i-2n+3)
      (i) -- (i-2n+5)
      (i) -- (i+2n-3)
      (i+2) -- (i-2n+5)
      (i+2) -- (i+2n-3)
;
\end{tikzpicture}\\
{\normalsize \it \enspace 2) Lemma~\ref{xy}, Item~\ref{xy+}: $b+c=2n+1$, $x\in\{0,1\}$ }\\~\\
\begin{tikzpicture}
\centering
\draw 
(4,0) node [draw,shape=circle,minimum width=19pt, color=blue](i){$x$} 
(4,-0.3)node[below] {\tiny $i$}
(5,0) node [draw,shape=circle,minimum width=19pt,color=red](i+2){$\overline x$} 
(5,-0.3)node[below] {\tiny $i+2$}
(2,2.5) node [draw,shape=circle, minimum width=19pt](i-2n+3){} 
(2.2,2.8)node[above] {\tiny $i-2n+3$}
(1,2.5) node [draw,shape=circle,minimum width=19pt,color=blue](i-2n+1){$x$} 
(0.8,2.8)node[above] {\tiny $i-2n+1$}
(8,2.5) node [draw,shape=circle,minimum width=19pt,color=red](i+2n+1){$\overline x$} 
(8.2,2.8)node[above] {\tiny $i+2n+1$}
(7,2.5) node [draw,shape=circle,minimum width=19pt](i+2n-1){} 
(6.8,2.8)node[above] {\tiny $i+2n-1$}
;
\draw[black, dashed]
 (6,2.5) node [draw,shape=circle,minimum width=19pt](i-2n+5){} 
  (3,2.5) node [draw,shape=circle,minimum width=19pt](i+2n-3){} 
;
\draw[loosely dotted] 
    (4.1,2.5) -- (4.9,2.5)
    (8.5,2.5) -- (8.9,2.5)
    (1.1,0) -- (1.9,0)
    (7.1,0) -- (7.9,0)
    (0.1,2.5) -- (0.5,2.5)
;
\draw [gray!130]
      (i+2) -- (i-2n+3)
      (i+2) -- (i+2n+1)
      (i+2) -- (i+2n-1)
      (i) -- (i-2n+1)
      (i) -- (i+2n-1)
      (i) -- (i-2n+3)
      (i) -- (i-2n+5)
      (i) -- (i+2n-3)
      (i+2) -- (i-2n+5)
      (i+2) -- (i+2n-3)
;
\end{tikzpicture}\\
{\normalsize \it \enspace 3)Lemma~\ref{xy}, Item~\ref{xy-}: $b+c=2n-1$, $x\in\{0,1\}$ }\\
\caption{Coloring patterns of the graph $\mathrm{Ci}_\infty(D_n)$ provided by Lemma~\ref{xy}. }
\label{Fig:lemmaxy}
\end{figure}

Let $G$ be an infinite circulant graph, $\varphi$ be its $2$-coloring, and $s$ be a positive integer. The sequence of vertices $\{{i+js}\}_{j\in\mathbb{Z}}$ of $G$ for an integer $i$ is called {\it $s$-chain}. If the inequality $\varphi({i+js})\neq\varphi({i+(j+1)s})$ holds for every $j$, then the sequence $\{{i+js}\}_{j\in\mathbb{Z}}$ is called an {\it alternating $s$-chain}.

\begin{lemma}\label{b+c+}
Let $n$, $b$ and $c$ be positive integers. Let the pair $(b,c)$ be admissible for the graph $\mathrm{Ci}_\infty(D_n)$. If $b+c=2n+1$, then every perfect non-bipartite $2$-coloring $\varphi$ corresponding to the pair $(b,c)$ {has the period length $2n+1$.}
\end{lemma}
\begin{proof}

Let us suppose that $b+c=2n+1$ and that there is a vertex $i$ of $\mathrm{Ci}_\infty(D_n)$ such that $\varphi(i)\neq\varphi(i+2n+1)$. Let $\varphi(i)=x\in\{0,1\}$, then $\varphi(i+2n+1)=1-x=\overline{x}$.
Let us consider the vertex $i+2$. It cannot be colored with $\overline{x}$, since that would contradict item~\ref{xy+} of  Lemma~\ref{xy}. 
Thus $\varphi(i+2)=x$. According to item~\ref{xy0} of  Lemma~\ref{xy} the vertex $(i-2n+1)$ is colored with $\overline{x}$. Following the same logic we obtain that the vertex $(i+2n-1)$ is colored with $\overline{x}$, and the vertex $(i+4n)$ is colored with $x$. 
Applying item~\ref{xy+} and item~\ref{xy0} of Lemma~\ref{xy} to  vertices $(i+2n-1)$, $(i+4n)$, we obtain equalities $\varphi(i+4n-2)=x$ and $\varphi(i+6n-1)=\overline{x}$. 

Alternatively applying item~\ref{xy+} and item~\ref{xy0} of Lemma~\ref{xy} to pairs of vertices $(i+(2n-1)j,i+2n+1+(2n-1)j)_{j\in\mathbb{N}}$ and $(i+2-(2n+1)j, i-2n-1-(2n+1)j)_{j\in\mathbb{N}}$, we obtain two alternating $(2n-1)$-chains $\{i+(2n-1)j\}_{j\in\mathbb{Z}}$ and $\{i+2+(2n-1)j\}_{j\in\mathbb{Z}}$. Two alternating $(2n-1)$-chains in $\mathrm{Ci}_\infty(D_n)$ built in described way are shown in Figure~\ref{Fig:-chain}. The edges colored with gray do not exist in the graph, they are shown by illustrative reasons and connect pair of vertices $(l, l+2n+1)_{l\in\mathbb{Z}}$.

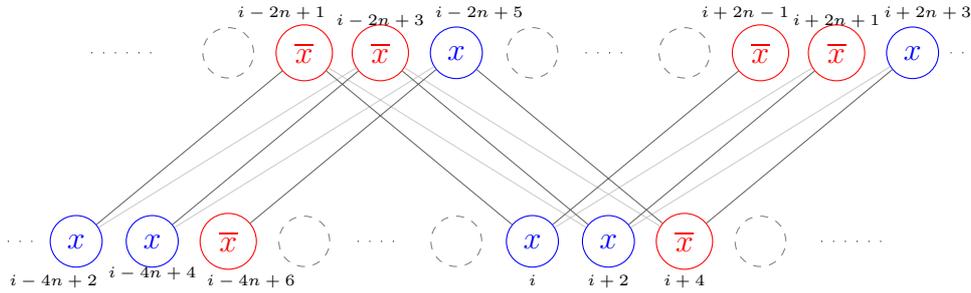
\begin{figure}[H]
\centering

\begin{tikzpicture}
\centering
\draw 
(-2,0) node [draw,shape=circle,color=blue,minimum width=19pt](i-4n+2){$x$} (-2.3,-0.3)node[below] {\tiny $i-4n+2$}
(-1,0) node [draw,shape=circle,color=blue,minimum width=19pt](i-4n+4){$x$} 
(-1,-0.2)node[below] {\tiny $i-4n+4$}
(4,0) node [draw,shape=circle,color=blue,minimum width=19pt](i){$x$} 
(4,-0.3)node[below] {\tiny $i$}
(5,0) node [draw,shape=circle,color=blue,minimum width=19pt](i+2){$x$} 
(5,-0.3)node[below] {\tiny $i+2$}
(6,0) node [draw,shape=circle,color=red,minimum width=19pt](i+4){$\overline x$} 
(6,-0.3)node[below] {\tiny $i+4$}
(2,2.5) node [draw,shape=circle,color=red,minimum width=19pt](i-2n+3){$\overline{x}$} 
(2,2.7)node[above] {\tiny $i-2n+3$}
(3,2.5) node [draw,shape=circle,color=blue,minimum width=19pt](i-2n+5){${x}$} 
(3.3,2.8)node[above] {\tiny $i-2n+5$}
(1,2.5) node [draw,shape=circle,color=red,minimum width=19pt](i-2n+1){$\overline{x}$} 
(0.7,2.8)node[above] {\tiny $i-2n+1$}
(8,2.5) node [draw,shape=circle,color=red,minimum width=19pt](i+2n+1){$\overline{x}$} 
(8,2.7)node[above] {\tiny $i+2n+1$}
(9,2.5) node [draw,shape=circle,color=blue,minimum width=19pt](i+2n+3){${x}$} 
(9.2,2.8)node[above] {\tiny $i+2n+3$}
(7,2.5) node [draw,shape=circle,color=red,minimum width=19pt](i+2n-1){$\overline{x}$} 
(6.8,2.8)node[above] {\tiny $i+2n-1$}
(0,0)node[draw,shape=circle,color=red,minimum width=19pt](i-4n+6){$\overline{x}$} 
(0.3,-0.3)node[below] {\tiny $i-4n+6$}

;
\draw[gray!100, dashed]
    (1,0) circle (9.5pt)
    (3,0) circle (9.5pt)
    (7,0) circle (9.5pt)
    (6,2.5) circle (9.5pt)
    (4,2.5) circle (9.5pt)
    (0,2.5) circle (9.5pt)
    
;
\draw[loosely dotted] 
    (4.7,2.5) -- (5.2,2.5)
    (9.5,2.5) -- (9.9,2.5)
    (1.7,0) -- (2.3,0)
    (7.8,0) -- (8.7,0)
    (-1.8,2.5) -- (-1,2.5)
    (-2.9, 0) -- (-2.5, 0)
;
\draw [gray!40]
	(i-4n+2)--(i-2n+3)
	(i-2n+3)--(i+4)
	(i-4n+4)--(i-2n+5)
	(i-2n+1)--(i+2)
	(i+2)--(i+2n+3)
	(i)--(i+2n+1)

;
           
\draw [gray!120] 
   (i-4n+2) -- (i-2n+1)
   (i-2n+1)--(i)
   (i)--(i+2n-1)
   (i-4n+4) -- (i-2n+3)
   (i-2n+3)--(i+2)
   (i+2)--(i+2n+1)
   (i-4n+6) -- (i-2n+5)
   (i-2n+5)--(i+4)
   (i+4)--(i+2n+3)
   
;
\end{tikzpicture}
\caption{Alternating $(2n-1)$-chains in $\mathrm{Ci}_\infty(D_n)$.}
\label{Fig:-chain}
\end{figure}

In view of two alternating chains shown in Figure~\ref{Fig:-chain} let us consider vertices $(i+4)$ and $(i+2n+3)$. 

{If $\varphi(i+4)=\overline{x}\neq\varphi(i+2)$, then, according to} item~\ref{xy+} of Lemma~\ref{xy}, the equality $\varphi(i+2n+3)=\varphi(i+2)$ holds. Let us note that there is no contradiction with inequality $\varphi(i+2)\neq\varphi(i-2n+3)$ obtained at the earlier steps of the construction process. Applying the same item to pairs of vertices $(i+2+(4n-2)j, i+4+(4n-2)j)_{j\in\mathbb{Z}})$ and $(i+2n+1+(4n-2)j, i+2n+3+(4n-2)j)_{j\in\mathbb{Z}}$ we obtain an alternative $(2n-1)$-chain $\{i+4+ (2n-1)j\}_{j\in\mathbb{Z}}$.

{Let us suppose that $\varphi(i+4)=x=\varphi(i+2)$. In this case} item~\ref{xy0} of Lemma~\ref{xy} gives the equality $\varphi(i+2n+3)=\varphi(i-2n+3)=\overline{x}$.
Applying the same item of Lemma~\ref{xy} to pairs of vertices $(i+2+(4n-2)j, i+4+(4n-2)j)_{j\in\mathbb{Z}})$ and $(i+2n+1+(4n-2)j, i+2n+3+(4n-2)j)_{j\in\mathbb{Z}}$ we obtain an alternative $(2n-1)$-chain $\{i+4+ (2n-1)j\}_{j\in\mathbb{Z}}$.

{If the color of the vertex $i+6$ is $\overline{x}=\varphi(i+4)$, then by }item~\ref{xy0} of Lemma~\ref{xy} the vertex $i+2n+5$ is colored with $\varphi(i-2n+5)=x$. Applying this item to the pairs of vertices $(i+4+(4n-2)j, i+6+(4n-2)j)_{j\in\mathbb{Z}}$ and $(i+2n+3+(4n-2)j, i+2n+5+(4n-2)j)_{j\in\mathbb{Z}}$ we obtain an alternative $(2n-1)$-chain $\{i+6+ (2n-1)j\}_{j\in\mathbb{Z}}$. {If the color of the vertex $i+6$ is $x\neq\varphi(i+4)$, then in accordance with} item~\ref{xy+} of Lemma~\ref{xy} the color of $i+2n+5$ is $\varphi(i+4)=\overline{x}$, the color of $i-2n+7$ is $\varphi(i-4n+6)=\overline{x}$, $\varphi(i-4n+8)=\varphi(i-6n+7)=x$, proceeding the same way we obtain an alternating $(2n-1)$-chain.

{Finally, the whole graph is colored with alternating $(2n-1)$-chains; the period length of the obtained coloring is $4n-2$, and the number of black and white vertices in the period is the same, meaning there are $2n-1$ edges with endpoints colored differently. If the obtained coloring is not bipartite, then this condition contradicts the} Remark~\ref{even}. Thus every perfect non-bipartite coloring corresponding to the case $b+c=2n+1$ is periodic with the period length $2n+1$.

\end{proof}

\begin{lemma}\label{b+c-}
Let $n$, $b$ and $c$ be positive integers. Let the pair $(b,c)$ be admissible for the graph $\mathrm{Ci}_\infty(D_n)$. If $b+c=2n-1$, then every perfect non-bipartite $2$-coloring $\varphi$ corresponding to the pair $(b,c)$ {has the period length $2n-1$.}
\end{lemma}
\begin{proof}
The proof of this lemma is similar to the previous one. First we suppose that there is a vertex $i$ such that $\varphi(i)\neq\varphi(i+2n-1)$.
According to item~\ref{xy-} of Lemma~\ref{xy} the vertex $i-2$ cannot be colored with $\overline{x}$, thus $\varphi(i-2)=x$ and $\varphi(i-2n-1)=\varphi(i+2n-1)=\overline{x}$. 
With the same logic $\varphi(i+2n+1)=\varphi(i+2n-1)=\overline{x}$ and $\varphi(i-4n)=\varphi(i)=x$; $\varphi(i-2n-3)=\varphi(i-2n-1)=\overline{x}$ and $\varphi(i-4n-2)=\varphi(i-2)=x$. Proceeding the same way we will obtain two alternative $(2n+1)$-chains, one is $\{i+(2n+1)j\}_{j\in\mathbb Z}$, another is $\{i-2+(2n+1)j\}_{j\in\mathbb{Z}}$.

The corresponding picture is shown in Figure~\ref{Fig:+chain}. Edges colored with gray represent parts of chains.

Let us consider the vertex $i+2$ and suppose that $\varphi(i+2)=x$.

The vertex $i+2n+3$ cannot be colored with $x$, because, provided with inequality $\varphi(i+2n+1)=\overline{x}\neq\varphi(i+2)$ it would contradict item~\ref{xy-} of Lemma~\ref{xy}, thus $\varphi(i+2n+3)=\varphi(i+2n+1)$. 
With the same logic $\varphi(i+2+(2n+1)j)=\varphi(i+(2n+1)j)$ for every $j\in\mathbb{Z}$, and finally we obtain an alternating chain $\{i+2+(2n+1)j\}_{j\in\mathbb{Z}}$.

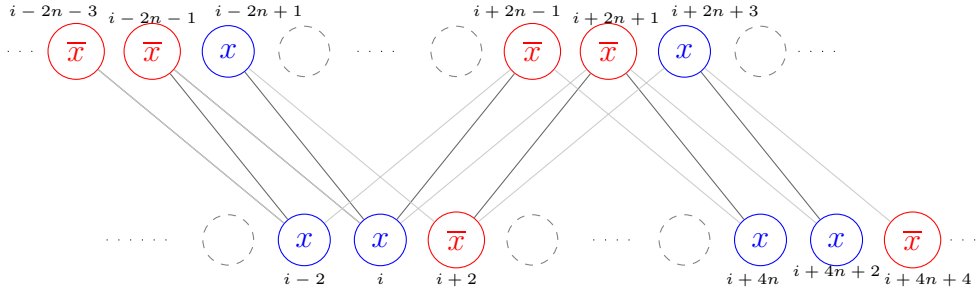
\begin{figure}[H]
\centering

\begin{tikzpicture}
\centering
\draw 
(-2,2.5) node [draw,shape=circle,color=red,minimum width=19pt](i-2n-3){$\overline x$}
(-2.3,2.8)node[above] {\tiny $i-2n-3$}
(-1,2.5) node [draw,shape=circle,color=red,minimum width=19pt](i-2n-1){$\overline x$} 
(-1,2.7)node[above] {\tiny $i-2n-1$}
(0,2.5) node [draw,shape=circle,color=blue,minimum width=19pt](i-2n+1){$ x$} 
(0.4,2.8)node[above] {\tiny $i-2n+1$}
(4,2.5) node [draw,shape=circle,color=red,minimum width=19pt](i+2n-1){$\overline x$} 
(3.8,2.8)node[above] {\tiny $i+2n-1$}
(5,2.5) node [draw,shape=circle,color=red,minimum width=19pt](i+2n+1){$\overline x$} 
(5.1,2.7)node[above] {\tiny $i+2n+1$}
(6,2.5) node [draw,shape=circle,color=blue,minimum width=19pt](i+2n+3){$ x$} 
(6.4,2.8)node[above] {\tiny $i+2n+3$}
(2,0) node [draw,shape=circle,color=blue,minimum width=19pt](i){${x}$} 
(2,-0.3)node[below] {\tiny $i$}
(3,0) node [draw,shape=circle,color=red,minimum width=19pt](i+2){$\overline{x}$} 
(3,-0.3)node[below] {\tiny $i+2$}
(1,0) node [draw,shape=circle,color=blue,minimum width=19pt](i-2){${x}$} 
(1,-0.3)node[below] {\tiny $i-2$}
(8,0) node [draw,shape=circle,color=blue,minimum width=19pt](i+4n+2){${x}$} 
(8,-0.2)node[below] {\tiny $i+4n+2$}
(9,0) node [draw,shape=circle,color=red,minimum width=19pt](i+4n+4){$\overline{x}$} 
(9.2,-0.3)node[below] {\tiny $i+4n+4$}
(7,0) node [draw,shape=circle,color=blue,minimum width=19pt](i+4n){${x}$} 
(6.9,-0.3)node[below] {\tiny $i+4n$}

;
\draw[gray!100, dashed]
    (0,0) circle (9.5pt)
    (4,0) circle (9.5pt)
    (6,0) circle (9.5pt)
    (7,2.5) circle (9.5pt)
    (3,2.5) circle (9.5pt)
    (1,2.5) circle (9.5pt)
;
\draw[loosely dotted] 
    (4.8,0) -- (5.4,0)
    (9.5,0) -- (9.9,0)
    (1.7,2.5) -- (2.2,2.5)
    (7.5,2.5) -- (8,2.5)
    (-1.6,0) -- (-0.8,0)
    (-2.9, 2.5) -- (-2.5, 2.5)
;
\draw [gray!120]
	(i-2)--(i-2n-3)
	(i)--(i+2n-1)
	(i+2n+1)--(i+4n)
	(i-2)--(i-2n-1)
	(i-2n-1)--(i)
	(i+2)--(i+2n+1)
	(i-2n+1)--(i)
	(i+2n+3)--(i+4n+2)

;
           
\draw [gray!40] 
	(i-2n-3)--(i-2)
	(i-2n-1)--(i)
	(i)--(i+2n+1)
	(i+2n+1)--(i+4n+2)
	(i-2)--(i+2n-1)
	(i+2n-1)--(i+4n)
	(i+2n+3)--(i+4n+4)
	(i+2)--(i+2n+3)
	(i+2)--(i-2n+1)

;
\end{tikzpicture}
\caption{Alternating $(2n+1)$-chains in $\mathrm{Ci}_\infty(D_n)$.}
\label{Fig:+chain}
\end{figure}

In the case $\varphi(i+2)=\overline{x}$ we use item~\ref{xy-} of Lemma~\ref{xy} to color $i-2n+1$ with $x$. 
Considering the equalities $\varphi(i+2n+1)=\varphi(i+2)=\overline{x}$ and $\varphi(i+4n+2)=x$ we color $i+2n+3$ with $x$ in accordance with the the same pattern of item~\ref{xy-}. Proceeding acting the same way with vertices $i+2+(2n+1)j$ for $j\in\mathbb{Z}$ we obtain an alternating $(2n+1)$-chain.
We can proceed the same way and color the graph with alternating $(2n+1)$-chains.

{The obtained coloring has the period length $4n+2$ with equal number of black and white vertices in the period, i.e. $2n+1$ edges having differently colored endpoints. This contradicts the necessary condition for the non-bipartite coloring to be perfect} (Remark~\ref{even}), thus the only possible period length for this case is $2n-1$.

\end{proof}

\begin{proof}[Proof of Theorem~\ref{thm}]

According to Lemma~\ref{bc}, the sum $b+c$ can be equal to $4n$, $2n$, $2n+1$ or $2n-1$.

The only possible perfect coloring corresponding to the admissible pair $(b,c)$ with $b+c=4n$ is bipartite, and its minimal period is $[12]$.
 
Let us consider the other possible values of the sum $b+c$.

\begin{enumerate}
\item Let $b+c=2n$. By Lemma~\ref{bc}, every perfect coloring corresponding to the pair $(b,c)$ is periodic with the period length $4n$ and the parameter matrix $M_{0}=\left (\begin{array}{cc} 2n-b & b\\ 2n-b & b \end{array}\right)$. 
{The form of the matrix implies that the color composition of each vertex is independent of its own color.
In this case any coloring with $b$ white and $2n-b$ black vertices provided with the condition from} Remark~\ref{even} is perfect.

Let us consider the graph $\mathrm{Ci}_{4n}(D_n)=(V,E)$ with $V=\{0,1,2,\dots, 4n-1\}$. It is the complete bipartite graph $K_{2n,2n}$, and thus for any its perfect $2$-coloring every vertex is adjacent to the same number of white and black vertices regardless of its own color. The parameter matrix of any such coloring is necessary of the form $M_{0}$ for the suitable parameter $b$.

Provided with Proposition~\ref{h} and the written above one can construct a surjective map from the set of perfect colorings of a finite circulant to the set of perfect colorings of an infinite one that maps a coloring $\phi$ to the coloring of $\mathrm{Ci}_\infty(D_n)$ with period $[\phi(0)\phi(1)\phi(2)\cdots\phi(4n-1)]$. It is easy to see that every coloring of infinite circulant can be considered as the one induced from the coloring of the finite graph. Let us note that different colorings of the finite circulant can induce the same coloring of the infinite circulant.

\item Let $b+c=2n+1$. According to Lemma~\ref{b+c+}, every perfect $2$-coloring corresponding to the pair has the period length $2n+1$ and the parameter matrix $M_{+1}=\left (\begin{array}{cc} 2n-b & b\\ 2n-b+1 & b-1 \end{array}\right)$. 

Let us consider the graph $\mathrm{Ci}_{4n+2}(D_n)$. The set of its perfect colorings is described in Subsection~\ref{4n+2}.{ In the case of two colors the non-bipartite construction requires each edge of the perfect matching $P_{2n+1}$ being monochrome. 
Let $\phi$ be a perfect $2$-coloring of the graph $\mathrm{Ci}_{4n+2}(D_n)$. It has parameter matrix $M_{+1}$ . Using the} Remark~\ref{even} and Proposition~\ref{h} {we can deduce that every perfect $2$-coloring of the infinite circulant is induced from a perfect coloring of $\mathrm{Ci}_\infty(D_n)$. The induced coloring of the infinite graph has the period $[\phi(0)\phi(1)\phi(2)\cdots\phi(4n+1)]$.}

\item Let $b+c=2n-1$. By Lemma~\ref{b+c-} every perfect $2$-coloring corresponding to the pair $(b,c)$ has the period length $2n-1$ and the parameter matrix $M_{-1}=\left (\begin{array}{cc} 2n-b & b\\ 2n-b-1 & b+1 \end{array}\right)$.

Let us consider the graph $\mathrm{Ci}_{4n-2}(D_n)$. The set of its perfect colorings is described in Subsection~\ref{4n-2}. {In the case of two colors and non-bipartite coloring every edge of the perfect matching $P_{2n-1}$ must be monochrome. Thus, such a coloring has the period length $2n-1$ and the parameter matrix of such a coloring is $M_{-1}$.
We can construct a surjective map from the set of perfect colorings of a finite circulant to the set of perfect colorings of an infinite one that maps a coloring $\phi$ to the coloring of $\mathrm{Ci}_\infty(D_n)$ with period $[\phi(0)\phi(1)\phi(2)\cdots\phi(4n-3)]$. It is easy to see that every coloring of infinite circulant can be considered as the one induced from the coloring of the finite graph.}

\end{enumerate}
\end{proof}

The main result of the paper confirms Conjecture~\ref{conj} in the case of two colors. 
{Let us note, that in this case the set of perfect colorings of the infinite path graph is a subset of perfect colorings induced from the colorings of the finite circulants $\mathrm{Ci}_t(D_n)$ for $t=4n-2,4n,4n+2$, and does not play a role in the colorings enumeration. Nevertheless, it will not be the case for a greater number of colors. For example, the coloring with the period $[1234567]$ is perfect for $\mathrm{Ci}_\infty(D_2)$, but is not perfect for any finite circulant $\mathrm{Ci}_t(D_2)$, $t=6,8,10$.}

We described how to construct perfect colorings of finite circulants from the conjecture, but the general question remains open. The main obstacle on the way of further classification of perfect colorings of infinite circulants is a large number of cases to study. The techniques of case reduction and examination of perfect colorings of such graphs are yet to be described. 

\section*{Acknowledgments}

Authors would like to thank Sergey~V.~Avgustinovich for helpful discussions and to express their gratitude to the anonymous reviewer for the careful reading of the manuscript and insightful comments and suggestions.

\end{document}